\documentclass{amsart}
\usepackage[all]{xy}
\usepackage{color}
\usepackage{stmaryrd}
\usepackage{amsthm}
\usepackage{amssymb}
\usepackage[colorlinks=true]{hyperref}
\usepackage{mathtools}

\numberwithin{equation}{subsection}

\newtheorem{theorem}{Theorem}[section]
\newtheorem*{theorem*}{Theorem}
\newtheorem{lemma}[theorem]{Lemma}

\newtheorem{proposition}[theorem]{Proposition}
\newtheorem{corollary}[theorem]{Corollary}
\newtheorem*{corollary*}{Corollary}

\theoremstyle{remark}
\newtheorem{definition}[theorem]{Definition}
\theoremstyle{remark}
\newtheorem{example}[theorem]{Example}

\theoremstyle{remark}
\newtheorem{remark}[theorem]{Remark}
\theoremstyle{remark}
\newtheorem{notation}[theorem]{Notation}

\DeclareMathOperator{\id}{id}
\DeclareMathOperator{\Id}{Id}

\newcommand{\too}{\longrightarrow}
\newcommand{\dg}{\mathrm{dg}}
\newcommand{\dgHo}{\mathrm{H}}

\newcommand{\cA}{{\mathcal A}}
\newcommand{\cB}{{\mathcal B}}
\newcommand{\cC}{{\mathcal C}}
\newcommand{\cD}{{\mathcal D}}
\newcommand{\cE}{{\mathcal E}}

\newcommand{\cO}{{\mathcal O}}

\newcommand{\cS}{{\mathcal S}}

\newcommand{\cX}{{\mathcal X}}

\newcommand{\bbA}{\mathbb{A}}

\newcommand{\bbL}{\mathbb{L}}
\newcommand{\bbH}{\mathbb{H}}

\newcommand{\bbP}{\mathbb{P}}
\newcommand{\bbK}{I\mspace{-6.mu}K}

\newcommand{\bbS}{\mathbb{S}}

\newcommand{\bbZ}{\mathbb{Z}}

\newcommand{\op}{\mathrm{op}} 
\newcommand{\ie}{\textsl{i.e.}\ }

\newcommand{\End}{\mathrm{End}}
\newcommand{\Nil}{\mathrm{Nil}}

\newcommand{\Hmo}{\mathrm{Hmo}}

\newcommand{\perf}{\mathrm{perf}} 

\newcommand{\Hom}{\mathrm{Hom}} 
\newcommand{\rep}{\mathrm{rep}} 
\newcommand{\dgcat}{\mathrm{dgcat}}

\newcommand{\Spt}{\mathrm{Spt}}

\newcommand{\internalcomment}[1]{}

\title[Algebraic $K$-theory with coefficients and Kleinian singularities]{$\bbA^1$-homotopy invariance of \\algebraic $K$-theory with coefficients\\ and Kleinian singularities}
\author{Gon{\c c}alo~Tabuada}

\address{Gon{\c c}alo Tabuada, Department of Mathematics, MIT, Cambridge, MA 02139, USA}
\email{tabuada@math.mit.edu}
\urladdr{http://math.mit.edu/~tabuada}

\date{\today}


\thanks{The author was partially supported by a NSF CAREER Award.}

\begin{document}
\begin{abstract}
C.~Weibel and Thomason-Trobaugh proved (under some assumptions) that algebraic $K$-theory  with coefficients is $\bbA^1$-homotopy invariant. In this article we generalize this result from schemes to the broad setting of dg categories. Along the way, we extend Bass-Quillen's fundamental theorem as well as Stienstra's foundational work on module structures over the big Witt ring to the setting of dg categories. Among other cases, the above $\bbA^1$-homotopy invariance result can now be applied to sheaves of (not necessarily commutative) dg algebras over stacks. As an application, we compute the algebraic $K$-theory with coefficients of dg cluster categories using solely the kernel and cokernel of the Coxeter matrix. This leads to a complete computation of the algebraic $K$-theory with coefficients of the Kleinian singularities parametrized by the simply laced Dynkin diagrams. As a byproduct, we obtain some vanishing and divisibility properties of algebraic $K$-theory (without coefficients).
\end{abstract}
\maketitle

\vskip-\baselineskip
\vskip-\baselineskip
\vskip-\baselineskip
\section{Introduction and statement of results}
Let $k$ be a base commutative ring, $X$ a quasi-compact quasi-separated $k$-scheme, and $l^\nu$ a prime power. As proved by Weibel in \cite[Page~391]{Web1} \cite[Thm.~5.2]{Web2} and by Thomason-Trobaugh in \cite[Thms.~9.5-9.6]{TT}, we have the following result:
\begin{theorem}\label{thm:main0}
\begin{itemize}
\item[(i)] When $ 1/l \in k$, the projection morphism $X[t] \to X$ gives rise to an homotopy equivalence of spectra $\bbK(X;\bbZ/l^\nu) \to \bbK(X[t];\bbZ/l^\nu)$.
\item[(ii)] When $l$ is nilpotent in $k$, the projection morphism $X[t] \to X$ gives rise to an homotopy equivalence of spectra $\bbK(X) \otimes \bbZ[1/l] \to \bbK(X[t])\otimes \bbZ[1/l]$.
\end{itemize}
\end{theorem}
The proof of Theorem \ref{thm:main0} is quite involved! The affine case, established by Weibel, makes use of a convergent right half-plane spectral sequence, of a universal coefficient sequence, of Bass-Quillen's fundamental theorem (see \cite[Page 236]{Grayson}), and more importantly of Stienstra's foundational work on module structures over the big Witt ring (see \cite[\S8]{Stienstra}). The extension to quasi-compact quasi-separated schemes, later established by Thomason-Trobaugh, is based on a powerful method known as ``reduction to the affine case''; consult \cite[\S9.1]{TT} for details.

\vspace{0.1cm}

The first goal of this article is to generalize Theorem \ref{thm:main0} from schemes to the broad setting of dg categories. Consult \S\ref{sec:applications}-\ref{sec:singularities} for applications and computations.
\subsection*{Statement of results}
A {\em differential graded (=dg) category $\cA$}, over the base commutative ring $k$, is a category enriched over complexes of $k$-modules; see \S\ref{sub:dg}. Every (dg) $k$-algebra $A$ gives naturally rise to a dg category with a single object. Another source of examples is provided by schemes since the category of perfect complexes $\perf(X)$ of every quasi-compact quasi-separated $k$-scheme $X$ admits a canonical dg enhancement $\perf_\dg(X)$; see \cite[\S4.4]{ICM-Keller}. Given a dg category $\cA$, let us write $\cA[t]$ for the tensor product $\cA\otimes k[t]$. Our first main result is the following:
\begin{theorem}\label{thm:main1}
\begin{itemize}
\item[(i)] When $ 1/l \in k$, the canonical dg functor $\cA\to \cA[t]$ gives rise to an homotopy equivalence of spectra $\bbK(\cA;\bbZ/l^\nu) \to \bbK(\cA[t];\bbZ/l^\nu)$.
\item[(ii)] When $l$ is nilpotent in $k$, the canonical dg functor $\cA \to \cA[t]$ gives rise to an homotopy equivalence of spectra $\bbK(\cA) \otimes \bbZ[1/l] \to \bbK(\cA[t])\otimes \bbZ[1/l]$.
\end{itemize}
\end{theorem}
Along the proof of Theorem \ref{thm:main1}, we adapt Bass-Quillen's fundamental theorem as well as Stienstra's foundational work on module structures over the big Witt ring to the broad setting of dg categories; see Theorems \ref{thm:fundamental} and \ref{thm:Witt}, respectively. These results are of independent interest. Except in Theorem \ref{thm:Witt}, we work more generally with a localizing invariant; see Definition \ref{def:localizing}. 
\section{Applications and computations}\label{sec:applications}
The second goal of this article is to explain how the above Theorem \ref{thm:main1} leads naturally to several applications and computations.
\subsection*{Sheaves of dg algebras}
Let $X$ be a quasi-compact quasi-separated $k$-scheme and $\cS$ a sheaf of (not necessarily commutative) dg $\cO_X$-algebras. Similarly to $\perf_\dg(X)$, we can consider the dg category $\perf_\dg(\cS)$ of perfect complexes of $\cS$-modules; see \cite[\S6]{Azumaya}. By applying Theorem \ref{thm:main1} to the dg category $\cA=\perf_\dg(\cS)$, we hence obtain the following generalization of Theorem~\ref{thm:main0}:
\begin{theorem}\label{thm:main2}
\begin{itemize}
\item[(i)] When $ 1/l \in k$, the projection morphism $\cS[t] \to \cS$ gives rise to an homotopy equivalence of spectra $\bbK(\cS;\bbZ/l^\nu) \to \bbK(\cS[t];\bbZ/l^\nu)$.
\item[(ii)] When $l$ is nilpotent in $k$, the projection morphism $\cS[t] \to \cS$ gives rise to an homotopy equivalence of spectra $\bbK(\cS) \otimes \bbZ[1/l] \to \bbK(\cS[t])\otimes \bbZ[1/l]$.
\end{itemize}
\end{theorem}
\begin{remark}[Orbifolds and stacks]
Given an orbifold, or more generally a stack $\cX$, we can also consider the associated dg category $\perf_\dg(\cX)$ of perfect complexes. Therefore, Theorem \ref{thm:main2} holds more generally for every sheaf $\cS$ of dg $\cO_\cX$-algebras.
\end{remark}
\subsection*{DG orbit categories}
Given a dg category $\cA$ and a dg functor $F: \cA \to \cA$ which induces an equivalence of categories $\dgHo^0(F): \dgHo^0(\cA) \stackrel{\simeq}{\to} \dgHo^0(\cA)$, recall from Keller \cite[\S5.1]{Orbit} the construction of the associated dg orbit category $\cA/F^\bbZ$. Thanks to Theorem \ref{thm:main1}, {\em all} the results established in \cite{dgOrbit} can now be applied to algebraic $K$-theory with coefficients. For example, \cite[Thm.~1.5]{dgOrbit} gives rise to the result:
\begin{theorem}\label{thm:main4}
When $1/l \in k$, we have a distinguished triangle of spectra:
$$ \bbK(\cA;\bbZ/l^\nu) \stackrel{\bbK(F;\bbZ/l^\nu) - \Id}{\too} \bbK(\cA;\bbZ/l^\nu) \too \bbK(\cA/F^\bbZ; \bbZ/l^\nu) \too \Sigma \bbK(\cA; \bbZ/l^\nu)\,.$$
When $l$ is nilpotent in $k$, the same holds with $\bbK(-;\bbZ/l^\nu)$ replaced by $\bbK(-)\otimes \bbZ[1/l]$.
\end{theorem}
\begin{remark}[Fundamental isomorphism]
When $F$ is the identity dg functor, the dg orbit category $\cA/F^\bbZ$ reduces to $\cA[t,1/t]$ and the above distinguished triangle splits. Consequently, we obtain a fundamental isomorphism between $\bbK(\cA[t,1/t]; \bbZ/l^\nu)$ and the direct sum $\bbK(\cA;\bbZ/l^\nu) \oplus \Sigma \bbK(\cA;\bbZ/l^\nu)$. When $l$ is nilpotent in $k$, the same holds with $\bbK(-;\bbZ/l^\nu)$ replaced by $\bbK(-)\otimes \bbZ[1/l]$. 
\end{remark}
\subsection*{DG cluster categories}
Let $k$ be an algebraically closed field, $Q$ a finite quiver without oriented cycles, $kQ$ the path $k$-algebra of $Q$, $\cD^b(kQ)$ the bounded derived category of finitely generated right $kQ$-modules, and $\cD_\dg^b(kQ)$ the canonical dg enhancement of $\cD^b(kQ)$. Consider the following dg functors
\begin{eqnarray*}
\tau^{-1} \Sigma^m : \cD^b_\dg(kQ) \too \cD^b_\dg(kQ) && m \geq 0\,,
\end{eqnarray*}
where $\tau$ is the Auslander-Reiten translation. Following Keller \cite[\S7.2]{Orbit}, the {\em dg $(m)$-cluster category $\cC_Q^{(m)}$ of $Q$} is defined as the dg orbit category $\cD_\dg^b(kQ)/(\tau^{-1} \Sigma^m)^\bbZ$. In the same vein, the {\em $(m)$-cluster category of $Q$} is defined as $\dgHo^0(\cC_Q^{(m)})$. These (dg) categories play nowadays a key role in representation theory of finite dimensional algebras; see Reiten's ICM-adress \cite{Reiten}. As proved by Keller-Reiten in \cite[\S2]{Acyclic}, the $(m)$-cluster categories (with $m \geq 1$) can be conceptually characterized as those $(m+1)$-Calabi-Yau triangulated categories containing a cluster-tilting object whose endomorphism algebra has a quiver without oriented cycles. 


As explained in \cite[Cor.~2.11]{dgOrbit}, in the particular case of dg cluster categories the above Theorem \ref{thm:main4} reduces to the following one:  

\begin{theorem}\label{thm:orbit}
When $l\neq \mathrm{char}(k)$, we have a distinguished triangle of spectra
$$ \bigoplus^v_{r=1} \bbK(k;\bbZ/l^\nu) \stackrel{(-1)^m \Phi_Q - \Id}{\too} \bigoplus^v_{r=1} \bbK(k;\bbZ/l^\nu) \to \bbK(\cC_Q^{(m)}; \bbZ/l^\nu) \to  \bigoplus^v_{r=1} \Sigma \bbK(k;\bbZ/l^\nu)\,,$$
where $v$ stands for the number of vertices of $Q$ and $\Phi_Q$ for the Coxeter matrix of $Q$. When $l = \mathrm{char}(k)$, the same holds with $\bbK(-;\bbZ/l^\nu)$ replaced by $\bbK(-)\otimes \bbZ[1/l]$.
\end{theorem}
As proved by Suslin in \cite[Cor.~3.13]{Suslin}, we have $\bbK_i(k;\bbZ/l^\nu) \simeq \bbZ/l^\nu$ when $i \geq 0$ is even and $\bbK_i(k;\bbZ/l^\nu) =0$ otherwise. Consequently, making use of the long exact sequence of algebraic $K$-theory groups with coefficients associated to the above distinguished triangle of spectra, we obtain the following result:
\begin{corollary}\label{cor:computation}
Consider the (matrix) homomorphism
\begin{equation}\label{eq:homo-key}
(-1)^m \Phi_Q - \Id : \bigoplus_{r=1}^v \bbZ/l^\nu \too \bigoplus_{r=1}^v \bbZ/l^\nu\,.
\end{equation}
When $l \neq \mathrm{char}(k)$, we have the following computation:
$$
\bbK_i(\cC^{(m)}_Q; \bbZ/l^\nu) \simeq \left\{ \begin{array}{lll}
\mathrm{cokernel}\,\,\mathrm{of}\,\, \eqref{eq:homo-key} & \mathrm{if} & i\geq 0\,\,\mathrm{even} \\
\mathrm{kernel}\,\,\mathrm{of}\,\,\eqref{eq:homo-key} & \mathrm{if} & i\geq 0 \,\,\mathrm{odd} \\
0 & \mathrm{if} & i<0 \,.
\end{array} \right.
$$
\end{corollary}
Note that Corollary \ref{cor:computation} provides a complete computation of the algebraic $K$-theory with coefficients of all dg orbit categories! Roughly speaking, all the information is encoded in the Coxeter matrix of the quiver. Note also that the kernel and cokernel of \eqref{eq:homo-key} have the same finite order. In particular, one is trivial if and only if the other one is trivial. Thanks to Corollary \ref{cor:computation}, this implies that the groups $\bbK_i(\cC^{(m)}_Q;\bbZ/l^\nu), i \geq 0$, are either all trivial or  all non-trivial.
\section{Kleinian singularities}\label{sec:singularities}
The third goal of this article is to explain how Corollary \ref{cor:computation} provide us a complete computation of the algebraic $K$-theory with coefficients of the Kleinian singularities. 

Let $k$ be an algebraically closed field of characteristic zero. Recall that the Kleinian singularities\footnote{Also usually known as {\em du Val} singularities.} are the isolated singularities of the singular affine hypersurfaces $R:=k[x,y,z]/(f)$ parametrized by the simply laced Dynkin~diagrams:
$$
\renewcommand{\arraystretch}{1.1}
 \begin{tabular}{| c | c | c | c | c | c |}
    \hline 
    type & $A_n, n \geq 1$ & $D_n, n \geq 4$ & $E_6$ & $E_7$ & $E_8$ \\ \hline
    $f$ & $x^{n+1} +yz$ & $x^{n-1}+ xy^2 + z^2$ & $x^4 + y^3 + z^2$ & $x^3y + y^3 + z^2$ & $x^5 + y^3 + z^2$ \\ \hline
  \end{tabular}
$$
Let $\underline{\mathrm{MCM}}(R)$ denote the stable category of maximal Cohen-Macaulay $R$-modules. Thanks to the work of Buchweitz \cite{Buchweitz} and Orlov \cite{Orlov,Orlov1}, this category is also usually known as the category of singularities $\cD^{\mathrm{sing}}(R)$ or equivalently as the category of matrix factorizations $\mathrm{MF}(k[x,y,z],f)$. Roughly speaking, $\underline{\mathrm{MCM}}(R)$ encodes all the information concerning the isolated singularity of the singular affine hypersurface~$R$.

Let $Q$ be a {\em Dynkin quiver}, \ie a quiver whose underlying graph is a Dynkin diagram of type $A$, $D$, or $E$. As explained by Keller in \cite[\S7.3]{Orbit}, $\underline{\mathrm{MCM}}(R)$ is equivalent to the category of finitely generated projective modules over the preprojective algebra $\Lambda(Q)$ and to the $(0)$-cluster category of $Q$. We hence conclude that the algebraic $K$-theory (with coefficients) of the Kleinian singularities is given by the algebraic $K$-theory (with coefficients) of the dg $(0)$-cluster categories $\cC^{(0)}_{A_n}, \cC^{(0)}_{D_n}, \cC^{(0)}_{E_6}, \cC^{(0)}_{E_7}, \cC^{(0)}_{E_8}$. In these cases, the homomorphisms \eqref{eq:homo-key} correspond to the following matrices (see \cite[Pages 289-290]{ARS}):
\begin{eqnarray*}
\xymatrix@C=1.5em@R=0.5em{A_n\colon 1 \ar[r] & 2 \ar[r] &  \cdots \ar[r] &  n-1 \ar[r] &n} && \begin{bmatrix}
-2 & 1 & 0  &\cdots &0 \\
-1 & -1 & \ddots &\ddots & \vdots\\
-1 & 0 & \ddots & \ddots & 0 \\
 \vdots& \vdots& \ddots & \ddots & 1 \\
 -1 &0 & \cdots& 0 & -1
\end{bmatrix}_{n \times n}
\end{eqnarray*}
\begin{eqnarray*}
\xymatrix@C=1.5em@R=0.2em{\quad\quad 1 \ar[dr] & & &   \\
D_n\colon\quad & 3 \ar[r] &  4 \ar[r] & \cdots \ar[r] & n& \\
\quad\quad 2 \ar[ur] & & & } && \begin{bmatrix}
-2 & 0 & 1 & 0 & \cdots&\cdots & 0 \\
0 & -2 & 1 & 0 & \ddots&\ddots & \vdots \\
-1& -1 & 0 & 1&\ddots& \ddots& \vdots\\
-1 & -1 & 1 & -1 & \ddots& \ddots & \vdots\\
 \vdots&\vdots & \vdots& 0 &\ddots & \ddots & 0 \\
 \vdots& \vdots&\vdots& \vdots& \ddots & \ddots & 1 \\
 -1 & -1 &1& 0 & \cdots & 0 & 1
\end{bmatrix}_{n \times n}
\end{eqnarray*}
\begin{eqnarray*}
\xymatrix@C=1.5em@R=1em{
& & 3 \ar[d] & &\\
E_6\colon 1 \ar[r] &2 \ar[r] & 4 \ar[r] & 5 \ar[r]& 6
 } && \begin{bmatrix}
-2 & 1 & 0 & 0 & 0 & 0 \\
-1 & -1 & 0& 1 & 0 & 0 \\
0 & 0 & -2 & 1& 0 & 0 \\
-2 & 0 & -1 & 0 & 1 & 0 \\
-1 & 0 & -1 & 1 & -1 & 1 \\
-1 & 0 & -1 & 1 & 0 & -1
\end{bmatrix}
\end{eqnarray*}
\begin{eqnarray*}
\xymatrix@C=1.5em@R=1em{
& & 3 \ar[d]& & & \\
E_7\colon 1 \ar[r] &2 \ar[r] & 4 \ar[r] & 5 \ar[r]& 6\ar[r] & 7
}&& \begin{bmatrix}
-2 & 1 & 0 & 0 & 0 & 0 & 0 \\
-1 & -1& 0 & 1 & 0 & 0 & 0 \\
0 &0& -2 & 1& 0 & 0 & 0 \\
-2 & 0 & -1& 0 & 1 & 0 & 0 \\
-1 & 0 & -1 &  1 & -1 & 1 & 0 \\
-1 & 0 & -1 & 1 & 0 & -1 & 1 \\
-1 & 0 & -1 & 1 & 0 & 0 & -1
\end{bmatrix}
\end{eqnarray*}
\begin{eqnarray*}
\xymatrix@C=1em@R=1em{
& & 3 \ar[d] & & & \\
E_8\colon 1 \ar[r] &2 \ar[r] & 4 \ar[r] & 5 \ar[r]& 6\ar[r] & 7\ar[r] & 8\\
}&& \begin{bmatrix}
-2 & 1 & 0 & 0 & 0 & 0 & 0 & 0 \\
-1 & -1 & 0 & 1 & 0 & 0 & 0 & 0 \\
0 & 0 & -2 & 1 & 0 & 0 & 0 & 0 \\
-1 & 0 & -1 & 0 & 1 & 0 & 0 & 0 \\
-1 & 0 & -1 &1 & -1 & 1 & 0 & 0\\
-1 & 0 & -1 & 1  & 0 & -1 & 1 & 0 \\
-1 & 0 & -1 & 1 & 0 & 0 & -1 & 1 \\
-1 & 0 & -1 & 1 & 0 & 0 & 0 & -1
\end{bmatrix}
\end{eqnarray*}
Thanks to Corollary \ref{cor:computation}, the computation of the algebraic $K$-theory with coefficients of the Kleinian singularities reduces then to the computation of the (co)kernels of the above explicit matrix homomorphisms! We now compute the type $A_n$ and leave the remaining cases to the reader. 
\begin{theorem}\label{thm:computation}
We have the following computation
$$
\bbK_i(\cC^{(0)}_{A_n}; \bbZ/l^\nu) \simeq \left\{ \begin{array}{lll}
\bbZ/\mathrm{gcd}(n+1,l^\nu)& \mathrm{if} & i\geq 0\\
0 & \mathrm{if} & i<0 \,,
\end{array} \right.
$$
where $\mathrm{gcd}(n+1,l^\nu)$ stands for the greatest common divisor of $n+1$ and $l^\nu$. Consequently, when $l| (n+1)$ (resp. $l\nmid (n+1)$) the abelian groups $\bbK_i(\cC^{(0)}_{A_n};\bbZ/l^\nu), i \geq 0$ (resp. $\bbK_i(\cC^{(0)}_{A_n};\bbZ/l^\nu), i \in \bbZ$), are non-trivial (resp. trivial).
\end{theorem}
Intuitively speaking, Theorem \ref{thm:computation} shows that the algebraic $K$-theory with $\bbZ/l^\nu$-coefficients of the isolated singularity of the affine hypersurface $k[x,y,z]/(x^{n+1} + yz)$ measures the $l$-divisibility of the integer $n+1$. To the best of the author's knowledge, these computations are new in the literature. They lead to the following vanishing and divisibility properties of algebraic $K$-theory (without coefficients):
\begin{corollary}\label{cor:K-theory}
\begin{itemize}
\item[(i)] For every $i \geq 0$, at least one of the following algebraic $K$-theory groups $\bbK_{i}(\cC^{(0)}_{A_n}), \bbK_{i-1}(\cC^{(0)}_{A_n})$ is non-trivial.
\item[(ii)] For every $l \nmid (n+1)$ the algebraic $K$-theory groups $\bbK_i(\cC^{(0)}_{A_n}), i \in \bbZ$, are uniquely $l$-divisible, \ie they are $\bbZ[1/l]$-modules.
\end{itemize}
\end{corollary}
Roughly speaking, Corollary \ref{cor:K-theory} shows that at least half of the groups $\bbK_i(\cC^{(0)}_{A_n})$ are non-trivial and moreover that they are ``large" from the divisibility viewpoint.
\begin{proof}
Consider the following universal coefficient sequences (see \S\ref{sec:proof}):
$$
0 \to \bbK_i(\cC_{A_n}^{(0)}) \otimes_\bbZ \bbZ/l \to \bbK_i(\cC_{A_n}^{(0)}; \bbZ/l) \to \{ l\text{-}\mathrm{torsion}\,\,\mathrm{in}\,\,\bbK_{i-1}(\cC^{(0)}_{A_n})\} \to 0\,.
$$
Let $l$ be a prime factor of $n+1$. Thanks to Theorem \ref{thm:computation}, the algebraic $K$-theory groups $\bbK_i(\cC^{(0)}_{A_n};\bbZ/l), i \geq 0$, are non-trivial. Therefore, item (i) follows from the above short exact sequences. Let $l$ be a prime number which does not divides $n+1$. Thanks to Theorem \ref{thm:computation}, the algebraic $K$-theory groups $\bbK_i(\cC^{(0)}_{A_n}; \bbZ/l), i \in \bbZ$, are trivial. Therefore, item (ii) follows also from the above short exact sequences.
\end{proof}
%

%
%
%
%
\begin{remark}[$K$-theory $\neq$ $G$-theory]
Recall from Orlov \cite{Orlov,Orlov} that the dg category of singularities $\cD^{\mathrm{sing}}_\dg(R)$ is defined as the Drinfeld's dg quotient (see \cite{Drinfeld}) of $\cD^b_\dg(R)$ by its full dg subcategory of perfect complexes. In the case where $R$ is the singular affine hypersurface $k[x,y,z]/(x^{n+1} + yz)$, this leads to a long exact sequence:
$$ \cdots \too \bbK_{i+1}(\cC^{(0}_{A_n}) \too \bbK_i(R) \too G_i(R) \too \bbK_i(\cC^{(0)}_{A_n}) \too \cdots\,.$$
Thanks to item (i) of the above Corollary \ref{cor:K-theory}, we hence conclude that the natural homomorphisms $K_i(R) \to G_i(R), i \geq 0$, are {\em not} isomorphisms.
\end{remark}
\subsection*{A cyclic quotient singularity}
Let the cyclic group $\bbZ/3$ act on the power series ring $k\llbracket x,y,z\rrbracket$ by multiplication with a primitive third root of unit. As proved by Keller-Reiten in \cite[\S2]{Acyclic}, the stable category of maximal Cohen-Macaulay modules $\underline{\mathrm{MCM}}(R)$ over the fixed point ring $R:=k\llbracket x,y,z\rrbracket^{\bbZ/3}$ is equivalence to the $(1)$-cluster category of the generalized Kronecker quiver $Q\!:\!\!\!\xymatrix@C=1.7em@R=1em{1\ar@<0.7ex>[r]\ar[r]\ar@<-0.7ex>[r] & 2}$. In this case the above homomorphism \eqref{eq:homo-key} is given by the matrix $\begin{bmatrix} -9 & 3 \\-3& 0\end{bmatrix}$. 
\begin{proposition}\label{prop:computation-last}
We have the following computation:
$$
\bbK_i(\cC^{(1)}_Q; \bbZ/l^\nu) \simeq \left\{ \begin{array}{lll}
\bbZ/3 \times \bbZ/3 & \mathrm{if} & i\geq 0\,\,\mathrm{and}\,\,l=3 \\
0&&\mathrm{otherwise} \,.
\end{array} \right.
$$
\end{proposition}
To the best of the author's knowledge, the above computation is new in the literature. Similarly to Corollary \ref{cor:K-theory}, for every $i \geq 0$ at least one of the following algebraic $K$-theory groups $\bbK_i(\cC^{(1)}_Q), \bbK_{i-1}(\cC^{(1)}_Q)$ is non-trivial, and that for every prime number $l\neq 3$ the groups $\bbK_i(\cC^{(1)}_Q), i \in \bbZ$, are uniquely $l$-divisible.
\medbreak\noindent\textbf{Acknowledgments:} The author is very grateful to Christian Haesemeyer for important discussions and comments on a previous version of the article. The author also would like to thank Lars Hesselholt, Niranjan Ramachandran, and Chuck Weibel for useful conversations.
\begin{remark}
After the circulation of this manuscript, Christian Haesemeyer kindly informed the author that some related computations concerning the $G$-theory of a local ring of finite Cohen-Macaulay type have been performed by Viraj Navkal \cite{Student}.
\end{remark}
\section{Preliminaries}
Throughout the article, $k$ will be a base commutative ring. Unless stated differently, all tensor products will be taken over $k$.
\subsection{Dg categories}\label{sub:dg}
Let $\cC(k)$ be the category of cochain complexes of $k$-modules. A {\em differential graded (=dg) category} $\cA$ is a $\cC(k)$-enriched category and a {\em dg functor} $F:\cA\to \cB$ is a $\cC(k)$-enriched functor; consult Keller's ICM survey \cite{ICM-Keller}. In what follows, $\dgcat(k)$ stands for the category of (small) dg categories and dg functors.

Let $\cA$ be a dg category. The category $\dgHo^0(\cA)$ has the same objects as $\cA$ and $\dgHo^0(\cA)(x,y):=H^0\cA(x,y)$. The dg category $\cA^\op$ has the same objects as $\cA$ and $\cA^\op(x,y):=\cA(y,x)$. A {\em right $\cA$-module} is a dg functor $M:\cA^\op \to \cC_\dg(k)$ with values in the dg category $\cC_\dg(k)$ of cochain complexes of $k$-modules. Let us write $\cC(\cA)$ for the category of right $\cA$-modules. 
The derived category $\cD(\cA)$ of $\cA$ is by definition the localization of $\cC(\cA)$ with respect to the (objectwise) quasi-isomorphisms. Its full triangulated subcategory of compact objects will be denoted by $\cD_c(\cA)$.

A dg functor $F:\cA\to \cB$ is called a {\em Morita equivalence} if it induces an equivalence of (triangulated) categories $\cD(\cA) \stackrel{\simeq}{\to} \cD(\cB)$; see \cite[\S4.6]{ICM-Keller}. As proved in \cite[Thm.~5.3]{IMRN}, $\dgcat(k)$ admits a Quillen model structure whose weak equivalences are the Morita equivalences. Let $\Hmo(k)$ be the associated homotopy category.

The tensor product $\cA\otimes\cB$ of dg categories is defined as follows: the set of objects is the cartesian product and $(\cA\otimes\cB)((x,w),(y,z)):= \cA(x,y) \otimes \cB(w,z)$. As explained in \cite[\S2.3 and \S4.3]{ICM-Keller}, this construction gives rise to symmetric monoidal categories $(\dgcat(k), -\otimes -, k)$ and $(\Hmo(k), -\otimes^\bbL- , k)$. 

An {\em $\cA\text{-}\cB$-bimodule} is a dg functor $\mathrm{B}:\cA \otimes^\bbL \cB^\op\to \cC_\dg(k)$ or equivalently a right $(\cA^\op \otimes^\bbL \cB)$-module. A standard example is the $\cA\text{-}\cB$-bimodule
\begin{eqnarray}\label{eq:bimodules111}
{}_F\mathrm{B}:\cA\otimes^\bbL \cB^\op \too \cC_\dg(k) && (x,w) \mapsto \cB(w,F(x))
\end{eqnarray}
associated to a dg functor $F:\cA \to \cB$. Finally, let us denote by $\rep(\cA,\cB)$ the full triangulated subcategory of $\cD(\cA^\op \otimes^\bbL \cB)$ consisting of those $\cA\text{-}\cB$-bimodules $\mathrm{B}$ such that $\mathrm{B}(x,-) \in \cD_c(\cB)$ for every object $x \in \cA$. 
\subsection{Exact categories}\label{sub:exact}
Let $\cE$ be an exact category in the sense of Quillen \cite[\S2]{Quillen1}. The following examples will be used in the sequel:
\begin{example}
Let $A$ be a $k$-algebra. Recall from \cite[\S2]{Quillen1} that the category $\mathrm{P}(A)$ of finitely generated projective right $A$-modules carries a canonical exact structure.
\begin{itemize}
\item[(i)] Let $\End(A)$ be the category of endomorphisms in $\mathrm{P}(A)$. The objects are the pairs $(M,f)$, with $M \in \mathrm{P}(A)$ and $f$ an endomorphism of $M$. The morphisms $(M,f) \to (M',f')$ are the $A$-linear maps $h:M \to M'$ such that $h f = f' h$. Note that $\End(A)$ inherits naturally from $\mathrm{P}(A)$ an exact structure making the forgetful functor $\End(A) \to \mathrm{P}(A), (M,f) \mapsto M$,~exact.
\item[(ii)]
Let $\Nil(A)$ be the category of nilpotent endomorphisms in $\mathrm{P}(A)$. By construction, $\Nil(A)$ is a full exact subcategory of $\End(A)$. 
\end{itemize}
\end{example}
Following Keller \cite[\S4.4]{ICM-Keller}, the {\em derived dg category $\cD_\dg(\cE)$ of $\cE$} is defined as the Drinfeld's dg quotient $\cC_\dg(\cE)/Ac_\dg(\cE)$ of the dg category of cochain complexes over $\cE$ by the full dg subcategory of acyclic complexes. 
\begin{notation}
Let $\cE_{{\bf dg}}$ denote the full subcategory of $\cD_\dg(\cE)$ consisting of those objects which belong to $\cD_c(\cE)$. By construction, we have $\dgHo^0(\cE_{\bf dg})\simeq \cD_c(\cE)$. Note that when $\cE=\mathrm{P}(A)$, we have a Morita equivalence between $\cE_{\bf dg}$ and $A$.
\end{notation}
Every exact functor $\cE \to \cE'$ gives rise to a dg functor $\cD_\dg(\cE) \to \cD_\dg(\cE')$ which restricts to $\cE_{{\bf dg}} \to \cE'_{{\bf dg}}$. In the same vein, every multi-exact functor $\cE \times \cdots \times \cE' \to \cE''$ gives rise to a dg functor $\cD_\dg(\cE) \otimes^\bbL \cdots \otimes^\bbL \cD_\dg(\cE') \to \cD_\dg(\cE'')$ which restricts to $\cE_{{\bf dg}} \otimes^\bbL \cdots \otimes^\bbL \cE'_{{\bf dg}} \to \cE''_{{\bf dg}}$.
\subsection{Algebraic $K$-theory with coefficients}\label{sub:K-theory}
Let $\Spt$ be the homotopy category of spectra and $\bbS$ the sphere spectrum. Recall from \cite[Thm.~10.3]{Duke} the construction of the nonconnective algebraic $K$-theory functor $\bbK:\dgcat(k) \to \Spt$. Given a prime power $l^\nu$, the algebraic $K$-theory with $\bbZ/l^\nu$-coefficients is defined as follows\footnote{Given any two prime numbers $p$ and $q$, we have $\bbS/pq \simeq \bbS/p \oplus \bbS/q$ in $\Spt$. Therefore, without loss of generality, we can (and will) work solely with one prime power $l^\nu$.}
\begin{eqnarray}\label{eq:coefficients}
\bbK(-;\bbZ/l^\nu): \dgcat(k) \too \Spt && \cA \mapsto \bbK(\cA) \wedge^\bbL \bbS/l^\nu\,,
\end{eqnarray} 
where $\bbS/l^\nu$ stands for the mod-$l^\nu$ Moore spectrum of $\bbS$. In the same vein, we have the functor $\bbK(-)\otimes \bbZ[1/l]: \dgcat(k) \to \Spt$ defined by the homotopy colimit
$$\bbK(\cA)\otimes \bbZ[1/l]:= \mathrm{hocolim}\left(\bbK(\cA) \stackrel{\cdot l}{\too} \bbK(\cA) \stackrel{\cdot l}{\too}\cdots  \right) \,.
$$
When $\cA=\perf_\dg(X)$, with $X$ a quasi-compact quasi-separated $k$-scheme, $\bbK(\cA)$ agrees with $\bbK(X)$; see \cite[\S5.2]{ICM-Keller}. Consequently, $\bbK(\cA;\bbZ/l^\nu)$ and $\bbK(\cA)\otimes \bbZ[1/l]$ agree with $\bbK(X;\bbZ/l^\nu)$ and $\bbK(X) \otimes \bbZ[1/l]$, respectively.
\subsection{Bass' construction}\label{sub:Bass}
Let $H: \dgcat(k) \to \mathrm{Ab}$ be a functor with values in the category of abelian groups. Following Bass \cite[\S XII]{Bass}, consider the sequence of functors $N^pH: \dgcat(k) \to \mathrm{Ab}, p \geq 0$, defined by $N^0H(\cA) := H(\cA)$ and by
\begin{eqnarray}\label{eq:formula}
& N^pH(\cA):=\mathrm{Kernel}\left(N^{p-1}H(\cA[t]) \stackrel{\id \otimes (t=0)}{\too} N^{p-1} H(\cA)\right) & p \geq 1\,.
\end{eqnarray}
Note that the canonical dg functor $\cA \to \cA[t]$ gives rise to a splitting $N^{p-1}H(\cA[t]) \simeq N^pH(\cA) \oplus N^{p-1}H(\cA)$. Let $\mathrm{Ch}^{\geq 0}(\bbZ)$ be the category of non-negatively graded chain complexes of abelian groups. Following Bass, we have also the functor  
\begin{eqnarray*}
N^\bullet H: \dgcat(k) \too \mathrm{Ch}^{\geq 0}(\bbZ) && \cA \mapsto N^\bullet H(\cA)\,,
\end{eqnarray*}
where the chain complex $N^\bullet H(\cA)$ is defined by $NH^0(\cA):= H(\cA)$ and by
$$
N^pH(\cA):= \bigcap_{i=1}^p \mathrm{Kernel}\left(H(\cA[t_1, \ldots, t_p]) \stackrel{\id \otimes (t_i=0)}{\too} H(\cA[t_1, \ldots, \widehat{t_i}, \ldots, t_p])\right)\,\,\, p \geq 1$$
\begin{eqnarray*}
N^pH(\cA) \too N^{p-1}H(\cA) && t_i \mapsto \left\{ \begin{array}{lcr}
1 - \sum_{i=2}^p t_i & \text{if} & i =1 \\
t_{i-1} & \text{if} & i \neq 1 \\
\end{array} \right.
\end{eqnarray*}
Note that the above two definitions of $N^pH(\cA)$ are isomorphic. In what follows we will simply write $NH(\cA)$ instead of $NH^1(\cA)$.

\section{Proof of Theorem \ref{thm:main1}}\label{sec:proof}
In what follows, we will work often with the following general notion:
\begin{definition}\label{def:localizing}
A functor $E: \dgcat(k) \to \Spt$ is called a {\em localizing invariant} if it inverts Morita equivalences and sends short exact sequences of dg categories (see \cite[\S4.6]{ICM-Keller}) to distinguished triangles of spectra:
\begin{eqnarray*}
0 \too \cA \too \cB \too \cC \too 0 & \mapsto & E(\cA) \too E(\cB) \too E(\cC) \stackrel{\partial}{\too} \Sigma E(\cA)\,.
\end{eqnarray*}
\end{definition}
Thanks to the work of Blumberg and Mandell, Keller, Schlichting, Thomason and Trobaugh, and others (see \cite{BM,Exact,Exact2,Schlichting,TT}), examples include not only nonconnective algebraic $K$-theory (with coefficients) but also Hochschild homology, cyclic homology, negative cyclic homology, periodic cyclic homology, topological Hochschild homology, topological cyclic homology, etc. Given an integer $q \in \bbZ$, the abelian group $\Hom_\Spt(\Sigma^q \bbS, E(\cA))$ will be denoted by $E_q(\cA)$. 
\subsection*{Step I - Spectral sequence}
Let $E:\dgcat(k) \to \Spt$ be a localizing invariant and $\Delta_n:=k[t_0, \ldots, t_n]/(\sum_{i=0}^n t_i -1), n \geq 0$, the simplicial $k$-algebra with faces and degeneracies given by the following formulas:
\begin{eqnarray*}
d_r(t_i) := \left\{ \begin{array}{lcr}
t_i & \text{if} & i <r \\
0 & \text{if} & i =r \\
t_{i-1} & \text{if} & i > r \\
\end{array} \right.
&
&
s_r(t_i) := \left\{ \begin{array}{lcr}
t_i & \text{if} & i <r \\
t_i + t_{i+1} & \text{if} & i =r \\
t_{i+1} & \text{if} & i > r \\
\end{array} \right.\,.
\end{eqnarray*}
Out of this data, we can construct the {\em $\bbA^1$-homotopization} of $E$:
\begin{eqnarray*}
E^h: \dgcat(k) \too \Spt && \cA \mapsto \mathrm{hocolim}_n E(\cA \otimes \Delta_n)\,.
\end{eqnarray*}
Note that $E^h$ comes equipped with a natural $2$-morphism $E \Rightarrow E^h$. As explained in \cite[Prop.~5.2]{A1-homotopy}, $E^h$ remains a localizing invariant and the canonical dg functor $\cA \to \cA[t]$ gives rise to an homotopy equivalence of spectra $E^h(\cA) \to E^h(\cA[t])$. 

Given an integer $q \in \bbZ$, consider the functor $E_q: \dgcat(k) \to \mathrm{Ab}$ and the associated non-negatively graded chain complex of abelian groups:
\begin{equation}\label{eq:complex}
0 \longleftarrow E_q(\cA) \stackrel{d_0 - d_1}{\longleftarrow} E_q(\cA[t]) \longleftarrow \cdots \stackrel{(-1)^r \sum_r d_r}{\longleftarrow} E_q(\cA\otimes \Delta_n) \longleftarrow \cdots\,.
\end{equation}
Under the following isomorphisms
\begin{eqnarray*}
\Delta_n \stackrel{\sim}{\to} k[t_1, \ldots, t_n] && t_i \mapsto \left\{ \begin{array}{lcr}
1 - \sum^n_{i=1} t_i  & \text{if} & i =0 \\
t_i& \text{if} & i \neq 0 \\
\end{array} \right.
\end{eqnarray*}
the (Moore) normalization of \eqref{eq:complex} identifies with $N^\bullet E_q(\cA)$. Consequently, following Quillen \cite{Quillen66}, we obtain a standard convergent right half-plane spectral sequence $E^1_{pq}= N^p E_q(\cA) \Rightarrow E_{p+q}^h(\cA)$. In the particular case of algebraic $K$-theory with coefficients, we hence have the following convergent spectral sequence:
\begin{equation}\label{eq:spectral1}
E^1_{pq}=N^p\bbK_q(\cA; \bbZ/l^\nu) \Rightarrow \bbK^h_{p+q}(\cA; \bbZ/l^\nu)
\end{equation}
Similarly, since $\pi_q(\bbK(\cA) \otimes \bbZ[1/l])\simeq \bbK_q(\cA)_{\bbZ[1/l]}$, we have the spectral sequence:
\begin{equation}\label{eq:spectral2}
E^1_{pq}=N^p \bbK_q(\cA)_{\bbZ[1/l]} \Rightarrow \bbK^h_{p+q}(\cA)_{\bbZ[1/l]}\,.
\end{equation}
\subsection*{Step II - Universal coefficient sequence}
Let $E:\dgcat(k) \to \Spt$ be a localizing invariant. Similarly to \eqref{eq:coefficients}, consider the following functor:
\begin{eqnarray*}
E(-;\bbZ/l^\nu): \dgcat(k) \too \Spt && \cA \mapsto E(\cA) \wedge^\bbL \bbS/l^\nu\,.
\end{eqnarray*}
For every dg category $\cA$ we have a distinguished triangle of spectra
\begin{equation}\label{eq:triangle-spectra}
E(\cA) \stackrel{\cdot l^\nu}{\too} E(\cA) \too E(\cA; \bbZ/l^\nu) \stackrel{\partial}{\too} \Sigma E(\cA)\,.
\end{equation}
Consequently, the associated long exact sequence (obtained by applying the functor $\Hom_{\Spt}(\bbS,-)$ to \eqref{eq:triangle-spectra}) breaks up into short exact sequences
$$
0 \to E_q(\cA) \otimes_\bbZ \bbZ/l^\nu \to E_q(\cA; \bbZ/l^\nu) \to \{l^\nu\text{-}\mathrm{torsion}\,\,\mathrm{in}\,\,E_{q-1}(\cA)\} \to 0 \,.
$$
Note that since the above distinguished triangle of spectra \eqref{eq:triangle-spectra} is functorial on $\cA$, we have moreover the following short exact sequences
\begin{equation*}
0 \to N^pE_q(\cA) \otimes_\bbZ \bbZ/l^\nu \to N^p E_q(\cA; \bbZ/l^\nu) \to \{l^\nu\text{-}\mathrm{torsion}\,\,\mathrm{in}\,\,N^pE_{q-1}(\cA)\} \to 0\,.
\end{equation*}
\subsection*{Step III - Fundamental theorem}
Recall that we have the exact functors:
\begin{eqnarray}
\Nil(k) \subset \End(k) \too \mathrm{P}(k) && (M,f) \mapsto M \label{eq:exact1} \\
\mathrm{P}(k) \too \Nil(k) \subset \End(k) && M \mapsto (M,0) \label{eq:exact2}\,.
\end{eqnarray}
Let $E:\dgcat(k) \to \Spt$ be a localizing invariant. Given a dg category $\cA$ and an integer $q \in \bbZ$, consider the following abelian group:
$$
\mathrm{Nil}E_q(\cA):=  \mathrm{Kernel}\left(E_q(\cA\otimes^\bbL \Nil(k)_{\bf dg}) \stackrel{\id \otimes \eqref{eq:exact1}}{\too} E_q(\cA\otimes^\bbL \mathrm{P}(k)_{\bf dg})\simeq E_q(\cA)\right)\,.$$
Note that since $\eqref{eq:exact1}\circ \eqref{eq:exact2}=\id$, the following morphism
\begin{equation}\label{eq:morphism-aux}
E(\cA) \simeq E(\cA\otimes^{\bbL} \mathrm{P}(k)_{\bf dg}) \stackrel{\id \otimes \eqref{eq:exact2}}{\too} E(\cA\otimes^\bbL \Nil(k)_{\bf dg})
\end{equation}
gives rise to a splitting $E_q(\cA\otimes^\bbL \Nil(k)_{\bf dg}) \simeq \Nil E_q(\cA) \oplus E_q(\cA)$.
\begin{theorem}[Fundamental theorem]\label{thm:fundamental}
We have $NE_{q+1}(\cA) \simeq \mathrm{Nil}E_q(\cA)$.
\end{theorem}
The remaining of this subsection is devoted to the proof of the above Theorem \ref{thm:fundamental}. Let $\bbP^1$ be the projective line over the base commutative ring $k$ and $i: \mathrm{Spec}(k[t]) \hookrightarrow \bbP^1$ and $j: \mathrm{Spec}(k[1/t]) \hookrightarrow \bbP^1$ the classical Zariski open cover.
\begin{proposition}\label{proof:exact-seq}
We have a short exact sequence of dg categories
\begin{equation}\label{eq:exact-1}
0 \too \Nil(k)_{\bf dg} \too \perf_\dg(\bbP^1) \stackrel{\bbL j^\ast}{\too} \perf_\dg(\mathrm{Spec}(k[1/t])) \too 0 \,. 
\end{equation}
\end{proposition}
\begin{proof}
Consider the following commutative diagram
$$
\xymatrix@C=1.3em@R=2em{
0 \ar[r] & \perf_\dg(\bbP^1)_Z \ar[d] \ar[r]& \perf_\dg(\bbP^1) \ar[d]_-{\bbL i^\ast} \ar[r]^-{\bbL j^\ast} & \perf_\dg(\mathrm{Spec}(k[1/t])) \ar[d] \ar[r] & 0 \\
0 \ar[r] & \perf_\dg(\mathrm{Spec}(k[t]))_{Z'} \ar[r] & \perf_\dg(\mathrm{Spec}(k[t])) \ar[r] & \perf_\dg(\mathrm{Spec}(k[t,1/t])) \ar[r] & 0 \,, 
}
$$
where $Z$ (resp. $Z'$) stands for the closed set $\bbP^1 - \mathrm{Spec}(k[1/t])$ (resp. $\mathrm{Spec}(k[t]) - \mathrm{Spec}(k[t,1/t])$) and $\perf_\dg(\bbP^1)_Z$ (resp. $\perf_\dg(\mathrm{Spec}(k[t]))_{Z'}$) for the dg category of those perfect complexes of $\cO_{\bbP^1}$-modules (resp. $\cO_{\mathrm{Spec}(k[t])}$-modules) which are supported on $Z$ (resp. on $Z'$). As proved by Thomason-Trobaugh in \cite[Thms.~2.6.3 and 7.4]{TT}, both rows are short exact sequences of dg categories and the left-hand side vertical dg functor is a Morita equivalence. 

Let us denote by $\bbH_{1,t}(k[t])$ the exact category of $t$-torsion $k[t]$-modules of projective dimension $\leq 1$. As proved by Grayson-Quillen in \cite[Page~236]{Grayson}, we have the following equivalence of exact categories
\begin{eqnarray}\label{eq:equivalence}
\Nil(k) \stackrel{\simeq}{\too} \bbH_{1,t}(k[t]) && (M,f) \mapsto M_f\,,
\end{eqnarray}
where $M_f$ stands for the $k[t]$-module $M$ on which $t$ acts as $f$. We hence obtain an induced Morita equivalence $\Nil(k)_{\bf dg} \simeq \bbH_{1,t}(k[t])_{\bf dg}$. The proof follows now from the Morita equivalence $\bbH_{1,t}(k[t])_{\bf dg}\simeq \perf_\dg(\mathrm{Spec}(k[t]))_{Z'}$ established by Thomason-Trobaugh in \cite[\S5.7]{TT}.
\end{proof}
As proved by Drinfeld in \cite[Prop.~1.6.3]{Drinfeld}, the functor $\cA\otimes^\bbL-: \Hmo(k) \to \Hmo(k)$ preserves short exact sequences of dg categories. Consequently, \eqref{eq:exact-1} gives rise to the following short exact sequence of dg categories
\begin{equation}\label{eq:exact-2}
0 \too \cA\otimes^\bbL\Nil(k)_{\bf dg} \too \cA \otimes^\bbL \bbP^1 \stackrel{\id \otimes \bbL j^\ast}{\too} \cA[1/t] \too 0 \,,
\end{equation}
where $\cA \otimes^\bbL\bbP^1$ stands for $\cA \otimes^\bbL \perf_\dg(\bbP^1)$. By applying the functor $E$ to \eqref{eq:exact-2}, we hence obtain a distinguished triangle of spectra 
\begin{equation}\label{eq:triangle}
E(\cA \otimes^\bbL \Nil(k)_{\bf dg}) \too E(\cA \otimes^\bbL \bbP^1) \too E(\cA[1/t]) \stackrel{\partial}{\too} \Sigma E(\cA \otimes^\bbL \Nil(k)_{\bf dg})\,.
\end{equation}
Now, recall from Thomason \cite[\S2.5-2.7]{Thomason} that we have two fully faithful dg functors 
\begin{eqnarray*}
\iota_{-1}: \perf_\dg(\mathrm{pt}) \too \perf_\dg(\bbP^1) && \cO_{\mathrm{pt}} \mapsto \cO_{\bbP^1}(-1) \\ \iota_0: \perf_\dg(\mathrm{pt}) \too \perf_\dg(\bbP^1) && \cO_{\mathrm{pt}} \mapsto \cO_{\bbP^1}(0)\,. 
\end{eqnarray*}
The dg functor $\iota_{-1}$ induces a Morita equivalence between $\perf_\dg(\mathrm{pt})$ and the Drinfeld's dg quotient $\perf_\dg(\bbP^1)/\iota_0 \perf_\dg(\mathrm{pt})$. Following \cite[\S13]{Duke}, we hence obtain a {\em split} short exact sequence of dg categories
\begin{equation}\label{eq:split}
\xymatrix{
0 \ar[r] & \perf_\dg(\mathrm{pt}) \ar[r]_-{\iota_0}  & \perf_\dg(\bbP^1) \ar[r]_s \ar@/_2ex/[l]_-r & \perf_\dg(\mathrm{pt}) \ar@/_2ex/[l]_-{\iota_{-1}}  \ar[r] & 0\,,
}
\end{equation}
where $r$ is the right adjoint of $\iota_0$, $r \circ \iota_0 =\id$, $\iota_{-1}$ is right adjoint of $s$, and $\iota_{-1} \circ s=\id$. By first applying the functor $\cA \otimes^\bbL - $ to \eqref{eq:split}, and then the functor $E$ to the resulting split short exact sequence of dg categories, we obtain the isomorphism 
\begin{equation}\label{eq:iso-1}
[E(\id \otimes \iota_0), E(\id \otimes \iota_{-1})]: E(\cA) \oplus E(\cA) \stackrel{\sim}{\too} E(\cA \otimes^\bbL \bbP^1)\,.
\end{equation}
The proof of the following general lemma is clear.
\begin{lemma}\label{lem:aux}
If $(f,g):A\oplus A \stackrel{\sim}{\to} B$ is an isomorphism in an additive category, then $(f,f-g):A\oplus A \stackrel{\sim}{\to} B$ is also an isomorphism.
\end{lemma}
By applying Lemma \ref{lem:aux} to \eqref{eq:iso-1}, we hence obtain the isomorphism 
\begin{equation}\label{eq:iso-2}
[E(\id \otimes \iota_0), E(\id \otimes \iota_0) - E(\id \otimes \iota_{-1})]: E(\cA) \oplus E(\cA) \stackrel{\sim}{\too} E(\cA\otimes^\bbL \bbP^1)\,.
\end{equation}
\begin{proposition}\label{prop:description}
The following composition
$$ E(\cA)\stackrel{\eqref{eq:morphism-aux}}{\too} E(\cA \otimes^\bbL \Nil(k)_{\bf dg}) \too E(\cA \otimes^\bbL \bbP^1)$$
agrees with $E(\id \otimes \iota_0) - E(\id \otimes \iota_{-1})$. 
\end{proposition}
\begin{proof}
As proved in \cite[Cor.~5.10]{IMRN}, there is a bijection between $\Hom_{\Hmo(k)}(\cA,\cB)$ and the set of isomorphism classes of the category $\rep(\cA,\cB)$. Under this bijection, the composition law of $\Hmo(k)$ corresponds to the bifunctor: 
\begin{eqnarray}\label{eq:bi-triang}
\rep(\cA,\cB) \times \rep(\cB,\cC) \too \rep(\cA,\cC) && (\mathrm{B},\mathrm{B}') \mapsto \mathrm{B} \otimes^\bbL_\cB \mathrm{B}'\,.
\end{eqnarray}
Since the $\cA\text{-}\cB$-bimodules \eqref{eq:bimodules111} belong to $\rep(\cA,\cB)$, we have the $\otimes$-functor:
\begin{eqnarray}\label{eq:functor-1}
\dgcat(k) \too \Hmo(k) & \cA \mapsto \cA & F \mapsto {}_F \mathrm{B}\,.
\end{eqnarray}
The {\em additivization $\Hmo_0(k)$ of $\Hmo(k)$} is the additive category with the same objects and abelian groups of morphisms given by $\Hom_{\Hmo_0(k)}(\cA,\cB):=K_0\rep(\cA,\cB)$, where $K_0\rep(\cA,\cB)$ stands for the Grothendieck group of the triangulated category $\rep(\cA,\cB)$. The composition law is induced by the above bitriangulated functor \eqref{eq:bi-triang} and the symmetric monoidal structure by bilinearity from $\Hmo(k)$. Note that we have also the following $\otimes$-functor:
\begin{eqnarray}\label{eq:functor-2}
\Hmo(k) \too \Hmo_0(k) & \cA \mapsto \cA & \mathrm{B} \mapsto [\mathrm{B}]\,.
\end{eqnarray}
Let us denote by $U:\dgcat(k) \to \Hmo_0(k)$ the composition of \eqref{eq:functor-1} with \eqref{eq:functor-2}. 

Now, consider the following composition of dg functors
$$\iota: \perf_\dg(\mathrm{pt}) \simeq \mathrm{P}(k)_{\bf dg} \stackrel{\eqref{eq:exact2}}{\too} \Nil(k)_{\bf dg} \to \perf_\dg(\bbP^1)\,.$$
Thanks to Proposition \ref{prop:factorization} below and to the fact that $U$ is a $\otimes$-functor, it suffices now to show that $U(\iota)$ agrees with $U(\iota_0) - U(\iota_{-1})$. As explained in \cite[Page~237]{Grayson}, we have a short exact sequence $0 \to \cO_{\bbP^1}(-1) \to \cO_{\bbP^1} \to \iota(\mathrm{pt})\to 0$. Consequently, we obtain a short exact of dg functors 
\begin{eqnarray*}
0 \to \iota_{-1} \to \iota_0 \to \iota \to 0 && \iota_{-1}, \iota_0, \iota: \perf_\dg(\mathrm{pt}) \too \perf_\dg(\bbP^1)\,.
\end{eqnarray*}
By construction of the additive category $\Hmo_0(k)$, we hence conclude that $[{}_{\iota}\mathrm{B}]=[{}_{\iota_0}\mathrm{B}]-[{}_{\iota_1}\mathrm{B}]$, \ie that $U(\iota)= U(\iota_0) - U(\iota_{-1})$. This achieves the proof.
\end{proof}
\begin{proposition}\label{prop:factorization}
Given a localizing invariant $E: \dgcat(k) \to \Spt$, there is an (unique) additive functor $\overline{E}: \Hmo_0(k) \to \Spt$ such that $\overline{E} \circ U \simeq E$.
\end{proposition}
\begin{proof}
Recall from \cite{IMRN} that a functor $E:\dgcat(k) \to \mathrm{D}$, with values in an additive category, is called an {\em additive invariant} if it inverts Morita equivalences and sends split short exact sequence of dg categories to direct sums. As proved in \cite[Thms.~5.3 and 6.3]{IMRN}, the functor $U:\dgcat(k) \to \Hmo_0(k)$ is the {\em universal additive invariant}, \ie given any additive category $\mathrm{D}$ there is an equivalence of categories
$$U^\ast: \mathrm{Fun}_{\mathrm{additive}}(\Hmo_0(k),\mathrm{D}) \stackrel{\simeq}{\too} \mathrm{Fun}_{\mathrm{add}}(\dgcat(k), \mathrm{D})\,,$$
where the left-hand side denotes the category of additive functors and the right-hand side the category of additive invariants. The proof follows now from the fact that every localizing invariant is in particular an additive invariant.
\end{proof}
The above distinguished triangle \eqref{eq:triangle} gives rise to the long exact sequence
$$
\cdots \to E_{q+1}(\cA \otimes^\bbL\bbP^1) \to E_{q+1}(\cA[1/t]) \to E_q(\cA \otimes^\bbL \Nil(k)_{\bf dg}) \to E_q(\cA\otimes^\bbL \bbP^1) \to \cdots
$$
Note that the following two compositions
\begin{equation}\label{eq:composition}
\xymatrix@C=3em@R=1em{
\perf_\dg(\mathrm{pt}) \ar@/^1ex/[r]^-{\iota_0} \ar@/_1ex/[r]_{\iota_{-1}} & \perf_\dg(\bbP^1) \ar[r]^-{\bbL j^\ast} & \perf_\dg(\mathrm{Spec}(k[1/t]))
}
\end{equation}
agree with the inverse image dg functor induced by $\mathrm{Spec}(k[1/t]) \to \mathrm{pt}$. Making use of the isomorphism \eqref{eq:iso-2}, we hence conclude that the above long exact sequence breaks up into shorter exact sequences
\begin{equation}\label{eq:short-3}
0 \to E_{q+1}(\cA) \to E_{q+1}(\cA[1/t]) \to E_q(\cA \otimes^\bbL \Nil(k)_{\bf dg}) \to E_q(\cA) \to 0 \,.
\end{equation}
Moreover, making use of Proposition \ref{prop:description}, we observe that the last morphism in \eqref{eq:short-3} corresponds to the projection $\mathrm{Nil}E_q(\cA) \oplus E_q(\cA) \to E_q(\cA)$. Consequently, \eqref{eq:short-3} can be further reduced to a short exact sequence 
$$ 0 \too E_{q+1}(\cA) \too E_{q+1}(\cA[1/t]) \too \mathrm{Nil} E_q(\cA) \too 0 \,.$$
From this short exact sequence we obtain finally the searched isomorphism
$$ NE_{q+1}(\cA) \simeq \mathrm{Cokernel}\left(E_{q+1}(\cA) \to E_{q+1}(\cA[1/t]) \right)\simeq \mathrm{Nil}E_q(\cA)\,.$$
\subsection*{Step IV - Module structure over the big Witt ring}
Given a commutative ring $R$, recall from Bloch \cite[Page~192]{Bloch} the construction of the big Witt ring $W(R)$. As an additive group, $W(R)$ is equal to $(1+t R\llbracket t \rrbracket, \times)$. The multiplication $\ast$ is uniquely determined by naturality, formal factorization of the elements of $W(R)$ as $h(t)=\prod_{n=0}^\infty (1-a_n t^n)$, and by the equality $(1-at) \ast h(t) = h(at)$. The zero element is $1+ 0t + \cdots$ and the unit element is $(1-t)$.
\begin{theorem}\label{thm:Witt}
Given a dg category $\cA$, the abelian groups $\mathrm{Nil} \bbK_q(\cA), q \in \bbZ$, carry a $W(k)$-module structure.
\end{theorem}
The remaining of this subsection is devoted to the proof of Theorem \ref{thm:Witt}. Recall that for every positive integer $n \geq 1$ we have a {\em Frobenius} functor 
\begin{eqnarray*}
F_n: \End(k) \too \End(k) && (M,f) \mapsto (M,f^n)
\end{eqnarray*}
as well as a {\em Verschiebung} functor
\begin{eqnarray*}
V_n: \End(k) \too \End(k) && (M,f) \mapsto (M^{\oplus n}, \begin{bmatrix} 0 & \cdots & \cdots & 0 &  f \\
1 & \ddots &  &  & 0 \\
0 & \ddots & \ddots &  & \vdots \\
\vdots & \ddots & \ddots & \ddots & \vdots\\
0 &\cdots & 0 & 1 & 0
 \end{bmatrix}_{n \times n}\!\!\!\!\!\!\!\!\!\!)\,.
\end{eqnarray*}
Both these functors are exact and preserve the full subcategory of nilpotent endomorphisms $\Nil(k)$. Moreover, the following diagrams are commutative:
\begin{equation}\label{eq:diagram-0}
\xymatrix{
\End(k) \ar[d]_-{\eqref{eq:exact1}} \ar[rr]^-{F_n} && \End(k) \ar[d]^-{\eqref{eq:exact1}} & \End(k) \ar[d]_-{\eqref{eq:exact1}} \ar[rr]^-{V_n} && \End(k) \ar[d]^-{\eqref{eq:exact1}} \\
\mathrm{P}(k) \ar@{=}[rr] && \mathrm{P}(k) & \mathrm{P}(k) \ar[rr]_-{M \mapsto M^{\oplus n}} && \mathrm{P}(k)\,.
}
\end{equation}
Let us denote by $\End_0(k)$ the kernel of the homomorphism $K_0\End(k) \stackrel{\eqref{eq:exact1}}{\too} K_0\mathrm{P}(k)$. Note that since $\eqref{eq:exact1}\circ \eqref{eq:exact2}=\id$, the homomorphism $K_0\mathrm{P}(k) \stackrel{\eqref{eq:exact2}}{\too} K_0\End(k)$ gives rise to a splitting $K_0\End(k) \simeq \End_0(k) \oplus K_0 \mathrm{P}(k)$. Note also that the image in $\End_0(k)$ of an endomorphism $(M,f)$ is given by $[(M,f)]-[(M,0)]$. Under these notations, we have induced Frobenius and Verschiebung homomorphisms $F_n, V_n: \End_0(k) \to \End_0(k)$. Consider also the following biexact~functor:
\begin{equation}\label{eq:biexact}
\End(k)\times \Nil(k) \too \Nil(k) \quad \quad ((M,f), (M',f')) \mapsto (M\otimes M', f \otimes f')
\end{equation}
and the associated commutative diagram:
\begin{equation}\label{eq:diagram-1}
\xymatrix{
\End(k)\times \Nil(k) \ar[d]_-{\eqref{eq:exact1}\times \eqref{eq:exact1}} \ar[rrr]^-{\eqref{eq:biexact}} & & & \Nil(k) \ar[d]^-{\eqref{eq:exact1}} \\
\mathrm{P}(k) \times \mathrm{P}(k) \ar[rrr]_-{(M,M') \mapsto M \otimes M'} &&& \mathrm{P}(k) \,.
}
\end{equation}
Given a dg category $\cA$, \eqref{eq:diagram-0} and \eqref{eq:diagram-1} give rise to the commutative diagrams:
$$
\xymatrix{
\cA\otimes^\bbL\Nil(k)_{\bf dg} \ar[d] \ar[r]^-{\id \otimes F_n} & \cA \otimes^\bbL \Nil(k)_{\bf dg} \ar[d] & \cA\otimes^\bbL \Nil(k)_{\bf dg} \ar[d]\ar[r]^-{\id \otimes V_n} & \cA\otimes^\bbL \Nil(k)_{\bf dg} \ar[d] \\
\cA\otimes^\bbL \mathrm{P}(k)_{\bf dg} \ar@{=}[r] & \cA\otimes^\bbL \mathrm{P}(k)_{\bf dg} & \cA\otimes^\bbL \mathrm{P}(k)_{\bf dg} \ar[r] & \cA\otimes^\bbL\mathrm{P}(k)_{\bf dg}\,.
}
$$
\begin{equation}\label{eq:diagram-dg}
\xymatrix{
\End(k)_{\bf dg}\otimes^\bbL \cA\otimes^\bbL \Nil(k)_{\bf dg} \ar[d] \ar[r]  & \cA\otimes^\bbL \Nil(k)_{\bf dg} \ar[d] \\
\mathrm{P}(k)_{\bf dg} \otimes^\bbL \cA \otimes^\bbL \mathrm{P}(k)_{\bf dg} \ar[r] & \cA\otimes^\bbL \mathrm{P}(k)_{\bf dg} \,.
}
\end{equation}
In what follows, we will still denote by $F_n, V_n: \Nil\bbK_q(\cA) \to \Nil \bbK_q(\cA)$ the induced Frobenius and Verschiebung homomorphisms. Since $\End_0(k)$ agrees with the kernel of the forgetful homomorphism $K_0(\End(k)_{\bf dg}) \stackrel{\eqref{eq:exact1}}{\too} K_0(\mathrm{P}(k)_{\bf dg})$, we hence obtain from \eqref{eq:diagram-dg} the bilinear pairings:
\begin{eqnarray}\label{eq:pairing}
- \cdot - : \End_0(k) \times \Nil \bbK_q(\cA) \too \Nil\bbK_q(\cA) && q \in \bbZ\,.
\end{eqnarray}
\begin{proposition}\label{prop:projection}
We have the following equality $V_n(\alpha \cdot F_n(\beta)) = V_n(\alpha) \cdot \beta$ for every $\alpha \in \End_0(k)$ and $\beta \in \Nil \bbK_q(\cA)$.
\end{proposition}
\begin{proof}
Let $S$ be the multiplicative closed subset of $\bbZ[x,y][s]$ generated by $s$ and $s^n - x^ny$. In what follows, we denote by $\End(\bbZ[x,y];S)$ the full exact subcategory of $\End(\bbZ[x,y])$ consisting of those endomorphisms $(N,g)$ for which there exists a polynomial\footnote{The polynomial $p(s)$ depends (a priori) on the endomorphisms $(N,g)$.} $p(s) \in S$ such that $p(g)=0$. For example, the endomorphisms
\begin{equation}\label{eq:matrix1}
\epsilon_1:= (\bbZ[x,y]^{\oplus n}, \begin{bmatrix} 0 &\cdots & \cdots& 0 & x^ny \\
1 & \ddots & &  & 0 \\
0 & \ddots & \ddots & & \vdots \\
\vdots& \ddots& \ddots & \ddots & \vdots \\
0 &  \cdots & 0 & 1 & 0 \end{bmatrix}_{n \times n}\!\!\!\!\!\!\!\!\!\!)
\end{equation}
\begin{equation}\label{eq:matrix2}
\epsilon_2:= (\bbZ[x,y]^{\oplus n}, \begin{bmatrix} 0 &\cdots & \cdots& 0 & xy \\
x & \ddots & &  & 0 \\
0 & \ddots & \ddots & & \vdots \\
\vdots& \ddots& \ddots & \ddots & \vdots \\
0 &  \cdots & 0 & x & 0 \end{bmatrix}_{n \times n}\!\!\!\!\!\!\!\!\!\!)
\end{equation}
belong to $\End(\bbZ[x,y];S)$ since they satisfy the equation $s^n - x^ny=0$. Following Stienstra \cite[\S5-6]{Stienstra}, consider the multiexact functor
$$
\theta(-,-,-): \End(\bbZ[x,y];S) \times \End(k) \times \Nil(k)  \too \Nil(k)
$$
which sends the triple $((N,g), (M,f), (M',f'))$ to the nilpotent endomorphism $(N \otimes_{\bbZ[x,y]}M\otimes M', g \otimes \id \otimes \id)$,  where the left $\bbZ[x,y]$-module structure on $M\otimes M'$ is given by $x \mapsto f'$ and $y \mapsto f$. Note that the following diagram commutes:
\begin{equation}\label{eq:diagram-3}
\xymatrix{
\End(\bbZ[x,y];S) \times \End(k) \times \Nil(k) \ar[d]_-{\eqref{eq:exact1}\times\eqref{eq:exact1} \times \eqref{eq:exact1}} \ar[rrr]^-{\theta(-,-,-)} &&& \Nil(k) \ar[d]^-{\eqref{eq:exact1}} \\
P(\bbZ[x,y]) \times \mathrm{P}(k) \times \mathrm{P}(k) \ar[rrr]_-{(N,M,M') \mapsto N\otimes_{\bbZ[x,y]}M\otimes M'} &&& \mathrm{P}(k)\,.
}
\end{equation}
Given a dg category $\cA$, \eqref{eq:diagram-3} gives rise to the commutative square:
\begin{equation}\label{eq:diagram5}
\xymatrix{
\End(\bbZ[x,y];S)_{\bf dg} \otimes^\bbL \End(k)_{\bf dg} \otimes^\bbL \cA \otimes^\bbL \Nil(k)_{\bf dg} \ar[r] \ar[d] & \cA\otimes^\bbL \Nil(k)_{\bf dg} \ar[d] \\
\mathrm{P}(\bbZ[x,y])_{\bf dg} \otimes^\bbL \mathrm{P}(k)_{\bf dg} \otimes^\bbL \cA \otimes^\bbL \mathrm{P}(k)_{\bf dg} \ar[r] & \cA \otimes^\bbL \mathrm{P}(k)_{\bf dg}\,.
}
\end{equation}
Consequently, we obtain from \eqref{eq:diagram5} the multilinear homomorphisms:
\begin{eqnarray}\label{eq:multi-2}
& \End_0(\bbZ[x,y];S) \times \End_0(k) \times \Nil \bbK_q (\cA) \too \Nil \bbK_q(\cA) & q \in \bbZ\,.
\end{eqnarray}
Thanks to Lemma \ref{lem:aux2} below, the evaluation of the above homomorphism \eqref{eq:multi-2} at the class $[\epsilon_1] - [(\bbZ[x,y]^{\oplus n}, 0)] \in \End_0(\bbZ[x,y];S)$ reduces to the bilinear pairing:
\begin{eqnarray}\label{eq:pairing1}
\End_0(k) \times \Nil \bbK_n(\cA) \too \Nil \bbK_n(\cA)&& (\alpha, \beta) \mapsto V_n(\alpha \cdot F_n(\beta))
\end{eqnarray}
Similarly, the evaluation of  \eqref{eq:multi-2} at $[\epsilon_2] - [(\bbZ[x,y]^{\oplus n}, 0)]$ reduces to the pairing:
\begin{eqnarray}\label{eq:pairing2}
\End_0(k) \times \Nil \bbK_q(\cA) \too \Nil \bbK_q(\cA) && (\alpha, \beta) \mapsto V_n(\alpha) \cdot \beta\,. 
\end{eqnarray}
Now, recall from Almkvist \cite{Almkvist} (see also \cite[Page 60]{Stienstra}) that the characteristic polynomial gives rise to an {\em injective} group homomorphism 
\begin{eqnarray*}
\End_0(\bbZ[x,y];S) \too W(\bbZ[x,y]) && ([(N,g)]-[(N,0)]) \mapsto \mathrm{det}(\id - gt)\,.
\end{eqnarray*}
Since the matrices \eqref{eq:matrix1}-\eqref{eq:matrix2} have the same characteristic polynomial, namely $1+ (x^ny)t^n$, we hence conclude that $[\epsilon_1]-[(\bbZ[x,y]^{\oplus n},0)]= [\epsilon_2]-[(\bbZ[x,y]^{\oplus n},0)]$. This implies that the above pairings \eqref{eq:pairing1}-\eqref{eq:pairing2} agree and consequently that $V_n(\alpha \cdot F_n(\beta)) = V_n(\alpha) \cdot \beta$ for every $\alpha \in \End_0(k)$ and $\beta \in \Nil \bbK_q(\cA)$.
\end{proof}
\begin{lemma}\label{lem:aux2}
We have the following commutative diagrams:
$$
\xymatrix{
\End(k) \times \Nil(k) \ar[d]_-{\id \times F_n} \ar[rr]^-{\theta(\epsilon_1, -,-)} && \Nil(k) &  \End(k) \times \Nil(k) \ar[d]_-{V_n \times \id} \ar[rr]^-{\theta(\epsilon_2, -,-)} && \Nil(k) \\
\End(k) \times \Nil(k) \ar[rr]_-{\eqref{eq:biexact}} && \Nil(k) \ar[u]_-{V_n} & \End(k) \times \Nil(k) \ar[rr]_-{\eqref{eq:biexact}} & & \Nil(k)  \ar@{=}[u] \,.
}
$$
\end{lemma}
\begin{proof}
Let $(M,f) \in \End(k)$ and $(M',f') \in \Nil(k)$. By definition of $\epsilon_1$ and $\epsilon_2$, we observe that $\theta(\epsilon_1, (M,f),(M',f'))$ is naturally isomorphic  to the endomorphism
$$((M\otimes M')^{\oplus n}, \begin{bmatrix} 0 & \cdots & \cdots & 0 & f\otimes f'^n \\
1 & \ddots & & & 0 \\
0 & \ddots & \ddots & & \vdots \\
\vdots & \ddots & \ddots & \ddots & \vdots \\
0 & \cdots & 0 & 1 & 0
\end{bmatrix}_{n\times n}\!\!\!\!\!\!\!\!\!\!)$$
and that $\theta(\epsilon_2, (M,f),(M',f'))$ is naturally isomorphic to the endomorphism 
$$
(M^{\oplus n} \otimes M', \begin{bmatrix} 0 & \cdots & \cdots & 0 & f \\
1 & \ddots & & & 0 \\
0 & \ddots & \ddots & & \vdots \\
\vdots & \ddots & \ddots & \ddots & \vdots \\
0 & \cdots & 0 & 1 & 0
\end{bmatrix}_{n \times n}\!\!\!\!\!\!\!\!\!\!\!\otimes f')\,.
$$
This achieves the proof.
\end{proof}
Given an integer $m \geq 0$, let $\Nil(k)^m$ be the full exact subcategory of $\Nil(k)$ consisting of those nilpotent endomorphisms $(M,f)$ such that $f^m=0$. By construction, we have an exhaustive increasing filtration $\Nil(k)^m \subset \Nil(k)^{m+1} \subset \cdots \subset \Nil(k)$. 

Given a dg category $\cA$ and an integer $q \in \bbZ$, let us denote by $\Nil \bbK_q(\cA)^m$ the image of the induced homomorphism:
$$
\mathrm{Kernel}\left(\bbK_q(\cA \otimes^\bbL \Nil(k)^m_{\bf dg}) \stackrel{\id \otimes \eqref{eq:exact1}}{\too} \bbK_q(\cA \otimes \mathrm{P}(k)_{\bf dg})\right) \too \Nil\bbK_q(\cA)\,.
$$
Note that $\Nil \bbK_q(\cA)$ identifies with $\bigcup_m \Nil \bbK_q(\cA)^m$ and that the Frobenius homomorphism $F_n: \Nil \bbK_q(\cA) \to \Nil \bbK_q(\cA)$ vanishes on $\Nil \bbK_q(\cA)^m$ whenever $n \geq m$. 

Given elements $a \in k$ and $\beta \in \Nil \bbK_q(\cA)$, consider the following definition
\begin{equation}\label{eq:Witt}
(1-a t^n) \odot \beta := V_n([(k,a)]-[(k,0)])\cdot \beta\,,
\end{equation}
where $(k,a)$ stands for the endomorphism of $k$ given by multiplication by $a$. Thanks to Proposition \ref{prop:projection}, \eqref{eq:Witt} agrees with $V_n(([(k,a)] -[(k,0)])\cdot F_n(\beta))$. Consequently, whenever $\beta \in \Nil \bbK_n(\cA)^m$ with $n \geq m$, we have $(1-a t^n) \odot \beta =0$. Since $\Nil \bbK_q(\cA)$ identifies with $\bigcup_m \Nil \bbK_q(\cA)^m$, we hence obtain the following $W(k)$-module structure on the abelian group $\Nil \bbK_q(\cA)$ (the sum is always finite!):
\begin{eqnarray*}
W(k) \times \Nil \bbK_q(\cA) \too \Nil \bbK_q(\cA) && (\prod_{n=0}^\infty (1-a_nt^n), \beta) \mapsto \sum_n ((1- a_nt^n) \odot \beta)\,.
\end{eqnarray*}
This concludes the proof of Theorem \ref{thm:Witt}.
\begin{remark}
Recall from Almkvist \cite{Almkvist} that the injective group homomorphism 
\begin{eqnarray*}
\End_0(k) \too W(k) && ([(M,f)]-[(M,0)]) \mapsto \mathrm{det}(\id - ft)
\end{eqnarray*}
sends $V_n([(k,a)]-[(k,0)])$ to $1 -a t^n$. This implies that the $W(k)$-module structure on the abelian group $\Nil \bbK_q(\cA)$ extends the above bilinear pairing \eqref{eq:pairing}.
\end{remark}
\subsection*{Proof of item (i)}
As explained by Weibel in \cite[Prop.~1.2]{Web2}, we have a ring homomorphism $\bbZ[1/l] \to W(\bbZ[1/l]), \lambda \mapsto (1-t)^\lambda$. Consequently, using the functoriality of $W(-)$ and the assumption $1/l \in k$, we observe that $W(k)$ is a $\bbZ[1/l^\nu]$-module. By combining Theorem \ref{thm:Witt} with Theorem \ref{thm:fundamental} (with $E=\bbK$), we hence conclude that the groups $N \bbK_q(\cA), q \in \bbZ$, carry a $\bbZ[1/l^\nu]$-module structure. The recursive formula \eqref{eq:formula} (with $H=\bbK_q$) implies that the groups $N^p \bbK_q(\cA), p \geq 1$, are also $\bbZ[1/l^\nu]$-modules. Therefore, making use of the short exact sequences (see Step II)
$$ 0 \to N^p\bbK_q(\cA) \otimes_\bbZ \bbZ/l^\nu \to N^p \bbK_q(\cA;\bbZ/l^\nu) \to \{l^\nu\text{-}\mathrm{torsion}\,\,\mathrm{in}\,\,N^p\bbK_{q-1}(\cA)\} \to 0\,,$$
we hence conclude that all the groups $N^p\bbK_q(\cA;\bbZ/l^\nu)$ are trivial. The convergent right half-plane spectral sequence \eqref{eq:spectral1} then implies that the edge morphisms $\bbK_q(\cA;\bbZ/l^\nu)\to \bbK_q^h(\cA;\bbZ/l^\nu)$ are isomorphisms. The proof follows now from the fact that the canonical dg functor $\cA \to \cA[t]$ gives rise to an homotopy equivalence of spectra $\bbK^h(\cA;\bbZ/l^\nu) \to \bbK^h(\cA[t]; \bbZ/l^\nu)$; see Step I.
\subsection*{Proof of item (ii)}
We start with the following (arithmetic) result:
\begin{lemma}\label{lem:l-group}
When $l$ is nilpotent in $k$, the abelian groups $\Nil \bbK_q(\cA)$ are $l$-groups.
\end{lemma}
\begin{proof}
Recall that the unit of $W(k)$ is $(1-t)$. Let $m \geq 0$ be a fixed integer. As explained by Weibel in \cite[\S1.5]{Web2}, whenever $l$ is nilpotent in $k$ there exists an integer $r \gg 0$ (which depends on $m$) such that $(1-t)^{l^r} \in 1 + t^m k\llbracket t\rrbracket$. This implies that the formal factorization of $(1-t)^{l^r}$ only contains factors $(1- a_n t^n)$ with $n \geq m$. As in Step IV, we hence observe that every element $\beta$ of $\Nil\bbK_q(\cA)^m$ is $l^r$-torsion. Finally, since $\Nil\bbK_q(\cA)\simeq \bigcup_m \Nil\bbK_q(\cA)^m$, we conclude that $\Nil \bbK_q(\cA)$ is a $l$-group.
\end{proof}
By combining Lemma \ref{lem:l-group} with Theorem \ref{thm:fundamental} (with $E=\bbK$), we conclude that the abelian groups $N\bbK_q(\cA), q \in \bbZ$, are $l$-groups. The recursive formula \eqref{eq:formula} (with $H=\bbK$) implies that the abelian groups $N^p\bbK_q(\cA), p \geq 1$, are also $l$-groups. Therefore, we have $N^p\bbK_q(\cA)_{\bbZ[1/l]}=0$. Making use of the convergent right half-plane spectral sequence \eqref{eq:spectral2}, we hence observe that the edge morphisms $\bbK_q(\cA)_{\bbZ[1/l]}\to \bbK_q^h(\cA)_{\bbZ[1/l]}$ are isomorphisms. The proof follows now from the fact that the dg functor $\cA \to \cA[t]$ gives rise to an homotopy equivalence of spectra $\bbK^h(\cA)\otimes \bbZ[1/l] \to \bbK^h(\cA[t])\otimes \bbZ[1/l]$; see Step I. 

\section{Proof of Theorem \ref{thm:computation}}
Thanks to Corollary \ref{cor:computation}, it suffices to compute the kernel and the cokernel of the (matrix) homomorphism \eqref{eq:homo-key} in the case where $m=0$ and $Q=A_n$. The kernel is the solution of the following system of linear equations with $\bbZ/l^\nu$-coefficients:
$$\begin{cases}
-2x_1 +x_2 =0 \\
-x_1 - x_2 + x_3 =0 \\
\quad \quad \quad \vdots \\
-x_1 - x_j + x_{j+1} =0 \\
\quad \quad \quad \vdots \\
-x_1 - x_{n_1} + x_n =0 \\
-x_1 -x_n =0\,.
\end{cases}
\Leftrightarrow \begin{cases}
x_2 =2x_1 \\
x_2 = -(n-1)x_1 \\
\quad \quad \vdots \\
x_j=-(n-j+1) x_1 \\
\quad \quad \vdots \\
x_{n-1}=-2x_1 \\
x_n = -x_1\,.
\end{cases} \Leftrightarrow
\begin{cases}
(n+1)x_1 =0 \\
\quad \quad  \vdots \\
x_j=-(n-j+1)x_1 \\
\quad \quad  \vdots \\
x_n=- x_1\,.
\end{cases} 
$$
From the above resolution of the system, we observe that the kernel is isomorphic to the $(n+1)$-torsion in $\bbZ/l^\nu$ or equivalently to the cyclic group $\bbZ/\mathrm{gcd}(n+1,l^\nu)$. Let us now compute the cokernel. Consider the following (matrix) homomorphism:
\begin{equation}\label{eq:auxiliar}
\begin{bmatrix}
-2 & 1 & 0  &\cdots &0 \\
-1 & -1 & \ddots &\ddots & \vdots\\
-1 & 0 & \ddots & \ddots & 0 \\
 \vdots& \vdots& \ddots & \ddots & 1 \\
 -1 &0 & \cdots& 0 & -1
\end{bmatrix}\colon\bigoplus^n_{r=1} \bbZ \too \bigoplus_{r=1}^n \bbZ
\end{equation}
Note that the cokernel of \eqref{eq:auxiliar} is isomorphic to $\bbZ/(n+1)$. A canonical generator is given by the image of the vector $(0,\cdots, 0, -1) \in \bigoplus^n_{r=1} \bbZ$. Using the fact that the functor $-\otimes_\bbZ \bbZ/l^\nu$ is left exact, we hence conclude that the cokernel of \eqref{eq:homo-key} is isomorphic to $\bbZ/(n+1) \otimes_\bbZ \bbZ/l^\nu \simeq \bbZ/ \mathrm{gcd}(n+1,l^\nu)$. This concludes the proof.
\begin{remark}\label{rk:Grothendieck}
Thanks to \cite[Cor.~2.11]{dgOrbit}, the Grothendieck group of $\cC_{A_n}^{(0)}$ identifies with the cokernel of \eqref{eq:auxiliar}. We hence observe that $K_0(\cC_{A_n}^{(0)})\simeq \bbZ/(n+1)$.
\end{remark}
\section{Proof of Proposition \ref{prop:computation-last}}
Similarly to the proof of Theorem \ref{thm:computation}, it suffices to compute the kernel and cokernel of the (matrix)  homomorphism \eqref{eq:homo-key} in the case where $m=1$ and $Q$ is the generalized Kronecker quiver $\xymatrix@C=1.7em@R=1em{1\ar@<0.7ex>[r]\ar[r]\ar@<-0.7ex>[r] & 2}$. The kernel is given by the solution of the following system of linear equations with $\bbZ/l^\nu$-coefficients:
\begin{equation}\label{eq:system}
\begin{cases}
-9x_1 +3x_2 =0 \\
-3x_1 =0 \,.
\end{cases}
\end{equation}
Clearly, the solution of \eqref{eq:system} is $(3\text{-}\mathrm{torsion}\,\,\mathrm{in}\,\,\bbZ/l^\nu) \times (3\text{-}\mathrm{torsion}\,\,\mathrm{in}\,\,\bbZ/l^\nu)$ or equivalently the cyclic group $\bbZ/\mathrm{gcd}(3,l^\nu) \times \bbZ/\mathrm{gcd}(3,l^\nu)$. Note that the latter group is isomorphic to $\bbZ/3 \times \bbZ/3$ when $l=3$ and is zero otherwise. Let us now compute the cokernel. Consider the following (matrix) homomorphism:
\begin{equation}\label{eq:square}
\begin{bmatrix} -9 & 3 \\-3& 0\end{bmatrix}: \bbZ \oplus \bbZ \too \bbZ \oplus \bbZ\,.
\end{equation}
The cokernel of \eqref{eq:square} is isomorphic to $\bbZ/3 \times \bbZ/3$. Canonical generators are given by the image of the vectors $(1,0)$ and $(-3,-1)$. Since the functor $-\otimes_\bbZ \bbZ/l^\nu$ is left exact, we hence conclude that the cokernel of \eqref{eq:homo-key} is isomorphic~to
$$
(\bbZ/3 \times \bbZ/3) \otimes_\bbZ \bbZ/l^\nu \simeq \bbZ/3\otimes_\bbZ \bbZ/l^\nu \times \bbZ/3 \otimes_\bbZ \bbZ/l^\nu \simeq \bbZ/\mathrm{gcd}(3,l^\nu) \times \bbZ/\mathrm{gcd}(3,l^\nu)\,.
$$
Once again, the right-hand side abelian group is isomorphic to $\bbZ/3 \times \bbZ/3$ when $l=3$ and is zero otherwise. This concludes the proof.
\begin{remark}
Similarly to Remark \ref{rk:Grothendieck}, the Grothendieck group of $\cC_Q^{(1)}$ identifies with the cokernel of \eqref{eq:square}. We hence observe that $K_0(\cC_Q^{(1)})\simeq \bbZ/3 \times \bbZ/3$.
\end{remark}


\begin{thebibliography}{00}

\bibitem{Almkvist} G.~Almkvist, {\em The Grothendieck ring of the category of endomorphisms}. Journal of Algebra {\bf 28} (1974), 375--388.

\bibitem{ARS} M.~Auslander, I.~ Reiten, and S.~Smal\o, {\em Representation theory of Artin algebras}. Cambridge Studies in Advanced Mathematics {\bf 36}. Cambridge University Press, Cambridge, 1995. 

\bibitem{Bass} H.~Bass, {\em Algebraic $K$-theory}. W. A. Benjamin, Inc., New York-Amsterdam 1968 xx+762 pp. 

\bibitem{Bloch} S.~Bloch, {\em Algebraic $K$-theory and crystalline cohomology}. 
Inst. Hautes {\'E}tudes Sci. Publ. Math. No. {\bf 47} (1977), 187--268 (1978).

\bibitem{BM} A. Blumberg and M. Mandell, {\em Localization theorems in topological Hochschild homology and
topological cyclic homology}. Geom. Topol. {\bf 16} (2012), no. 2, 1053--1120.

\bibitem{Buchweitz} R.~O.~Buchweitz, {\em Maximal Cohen-Macaulay modules and Tate-cohomology over Gorenstein rings}. Available at {\tt https://tspace.library.utoronto.ca/handle/1807/16682}.

\bibitem{Drinfeld} V.~Drinfeld, {\em DG quotients of DG categories}. Journal of Algebra {\bf 272} (2004), 643--691.

\bibitem{Grayson} D.~Grayson, {\em Higher algebraic $K$-theory $II$ (after Daniel Quillen).} Algebraic $K$-theory, Evanston 1976, Springer, Berlin, Heidelberg, New York, LNM {\bf 551} (1976), 217--240.

\bibitem{ICM-Keller} B.~Keller, {\em On differential graded
    categories}. International Congress of Mathematicians (Madrid), Vol.~II,
  151--190, Eur.~Math.~Soc., Z{\"u}rich, 2006.
  
\bibitem{Orbit} \bysame, {\em On triangulated orbit categories}. Doc. Math. {\bf 10} (2005), 551--581.  
  
\bibitem{Exact} \bysame, {\em On the cyclic homology of exact categories}. JPAA {\bf 136} (1999), no.~1, 1--56.   

\bibitem{Exact2} \bysame, {\em On the cyclic homology of ringed spaces and schemes}. Doc. Math. {\bf 3} (1998), 231--259.
  
\bibitem{Acyclic} B.~Keller and I.~Reiten, {\em Acyclic Calabi-Yau categories}. With an appendix by Michel Van den Bergh. Compos. Math. {\bf 144} (2008), no. 5, 1332--1348. 

\bibitem{Student} V.~Navkal, {\em $K'$-theory of a local ring of finite Cohen-Macaulay type}. Journal of $K$-Theory {\bf 12} (2013), no.~3, 405--432.

\bibitem{Orlov} D.~Orlov, {\em Derived categories of coherent sheaves and triangulated categories of singularities}. Algebra, arithmetic, and geometry: in honor of Yu. I. Manin. Vol. II, 503--531, 
Progr. Math., {\bf 270}, BirkhŠuser Boston, Inc., Boston, MA, 2009.

\bibitem{Orlov1} \bysame, {\em Triangulated categories of singularities and $D$-branes in Landau-Ginzburg models}. Proc. Steklov Inst. Math. {\bf 246} (2004), 227--248.

\bibitem{Quillen1} D.~Quillen, {\em Higher algebraic K-theory. I.} Algebraic K-theory, I: Higher K-theories (Proc. Conf., Battelle Memorial Inst., Seattle, Wash., 1972), 85--147. LNM {\bf 341}, Springer 1973. 

\bibitem{Quillen66} \bysame, {\em Spectral sequences of a double semi-simplicial group}. Topology {\bf 5} (1966), 155--157. 

\bibitem{Reiten} I.~Reiten, {\em Cluster categories}. Proceedings of the International Congress of Mathematicians. Volume {\bf I}, 558--594, Hindustan Book Agency, New Delhi, 2010.

\bibitem{Schlichting} M.~Schlichting, {\em Negative $\mbox{K}$-theory of
    derived categories}. Math.~Z. {\bf 253} (2006), no.~1, 97--134.

\bibitem{Stienstra} J.~Stienstra, {\em Operations in the higher $K$-theory of endomorphisms}. Current trends in algebraic topology, Part 1 (London, Ont., 1981), volume {\bf 2} of CMS Conf. Proc. 59--115. Amer. Math. Soc., Providence, R.I., 1982.

\bibitem{Suslin} A.~Suslin, On the $K$-theory of local fields. JPAA {\bf 34} (1984), no.~2-3, 301--318.

\bibitem{Duke} G.~Tabuada, {\em Higher $K$-theory via universal invariants}. Duke Mathematical Journal {\bf 145} (2008), no.~1, 121--206.

\bibitem{IMRN} \bysame, {\em Invariants additifs de dg-cat{\'e}gories}.
Int.~Math.~Res.~Not. {\bf 53} (2005), 3309--3339.

\bibitem{A1-homotopy} \bysame, {\em $\bbA^1$-homotopy theory of noncommutative motives}. Available at arXiv:1402.4432. To appear in Journal of Noncommutative Geometry.  

\bibitem{dgOrbit}  \bysame, {\em $\bbA^1$-homotopy invariants of dg orbit categories}. Available at arXiv:1409.8241v2.

\bibitem{Azumaya} G.~Tabuada and M. Van den Bergh,  {\em Noncommutative motives of Azumaya algebras}. Journal of the Institute of Mathematics of Jussieu {\bf 14} (2015), no. 2, 379--403.

\bibitem{Thomason} R.~W.~Thomason, {\em Les $K$-groupes d'un sch{\'e}ma {\'e}clat{\'e} et une formule d'intersection exc{\'e}dentaire}. Invent. Math. {\bf 112} (1993), no.~1, 195--215.

\bibitem{TT} R.~W.~Thomason and T.~Trobaugh, {\em Higher algebraic $K$-theory of schemes and of derived  categories}. Grothendieck Festschrift, Volume III. Volume {\bf 88} of Progress in Math., 247--436. Birkhauser, Boston, Bassel, Berlin, 1990. 
    
\bibitem{Web1} C.~Weibel, {\em Mayer-Vietoris sequences and mod $p$ $K$-theory}. LNM {\bf 966} (1983), 390--407.

\bibitem{Web2}  \bysame, {\em Mayer-Vietoris sequences and module structures on $NK^\ast$}. Lecture Notes in Mathematics {\bf 854} (1981), 466--493.
    
\end{thebibliography}
\end{document}